\documentclass{amsart}
\usepackage{txfonts}
\usepackage{amsfonts}

\newtheorem{theorem}{Theorem}[section]

\newtheorem{proposition}[theorem]{Proposition}
\newtheorem{corollary}[theorem]{Corollary}

\theoremstyle{definition}
\newtheorem{definition}[theorem]{Definition}
\newtheorem{example}[theorem]{Example}

\theoremstyle{remark}

\numberwithin{equation}{section}

\begin{document}

\title{Degrees of periods }

\author{}
\address{Center of Mathematical Sciences
Zhejiang University Hangzhou, Zhejiang, 310027, China}
\curraddr{} \email{wanj\_m@yahoo.com.cn}
\thanks{}

\author{Jianming Wan}
\address{}
\email{}
\thanks{}


\date{}

\dedicatory{}


\begin{abstract}
We introduce the concept of degree to classify the periods in the
sense of Kontsevich and Zagier. Some properties of degree are
proved. Using this notion we give some new understanding of some
problems in transcendental number theory.  The zeta function of a
period is defined and some its interesting properties are given.
\end{abstract}

\maketitle

\section*{}

\specialsection*{}



\section{Introduction}

In the wonderful exposition [2], Kontsevich and Zagier defined the
concept of period: integral of a rational function over a domain
bounded by polynomial inequalities with rational coefficients. By
its definition the set of periods is countable and includes all
algebraic numbers. Moreover, it is a ring, the sum and product of
two periods are still periods. Many important transcendental numbers
arising from modular forms, L-functions, hypergeometric functions,
etc are periods. On the other hand, from the point of view of
algebraic geometry, periods are integrals of closed algebraic
differential forms over relative algebraic chains (cf.[1] and [2]).

The Galois theory plays a fundamental role in algebraic number
theory. What can we do something for transcendental number theory?
From Grothendieck's motive point of view, period is a suitable
category for building a Galois theory (called motive Galois group) (
cf. [1]).

The periods are also intended to bridge the gap between the
algebraic numbers and the transcendental numbers. They are natural
objects whether from the point of view of number theory or algebraic
geometry.

The main purpose of the paper is to try to classify these periods
under suitable category. The main tool is the concept of degree
introduced by the author. We find that this concept can give some
theoretic solutions to some problems in transcendental number
theory. For example, we prove that the sum of two transcendental
periods with different degrees is a transcendental number. We also
define the zeta function for a period and prove some interesting
properties.

\section{definition of a period}

Let us recall the definition of a period [2].

\begin{definition}
A period is a complex number whose real and imaginary parts are
absolutely convergent multiple integrals $$\int_{\Sigma}R$$ where
$\Sigma$ is a domain in $\mathbb{R}^{n}$ given by polynomial
inequalities with rational coefficients and $R$ is a rational
function with rational coefficients.
\end{definition}

In above definition one can replace ''rational coefficients'' by
''algebraic coefficients'' by introducing more variables. Because
the integral of any real function is equal to the area under its
graph, any period can be written as the volume of a domain defined
by polynomial inequalities with rational coefficients. So we can
rewrite the definition as

\begin{definition}
A period is a complex number whose real and imaginary parts are
absolutely convergent multiple integrals
$$\int_{\Sigma}dx_{1}...dx_{n}$$ where $\Sigma$ is a domain in
$\mathbb{R}^{n}$ given by polynomial inequalities with algebraic
coefficients.
\end{definition}

For simplicity, in what follows we always use definition 2.2 as the
definition of a period.

The set of periods is clearly countable. It is a ring and includes
all algebraic numbers. For instance, let $p$ be an algebraic number,
then $$p=\int_{0\leq x\leq p }dx.$$

Many interesting transcendental numbers also are periods.

\begin{example}
\begin{enumerate}
\item $$\pi=\iint_{x^{2}+y^{2}\leq1}dxdy.$$

\item $$\log(q)=\iint_{1\leq x\leq q,xy\leq1,y\geq0}dxdy,$$ where $q$
is a positive algebraic number.

\item All $\zeta(s)$ ($s$ is positive integers) are periods [2].
$\zeta(s)$ is Riemann zeta function
$$\zeta(s)=\sum_{n=1}^{+\infty}\frac{1}{n^{s}}.$$ Recall that (cf. [3])
$\zeta(2k)=\frac{2^{2k-1}}{(2k)!}B_{k}\pi^{2k}$ where $B_{k}$ is the
Bernoulli number.

\item Some values of the gamma function
$$\Gamma(s)=\int_{0}^{\infty}t^{s-1}e^{-t}dt$$ at rational values,
$\Gamma(p/q)^{q}$ ($p,q\in\mathbb{N}$) are periods [2].

\item Let $$E_{k}(z)=\frac{1}{2}\sum_{m,n\in\mathbb{Z};
(m,n)=1}\frac{1}{(mz+n)^{k}}$$ be the Eisenstein series of weight
$k$. If $z_{0}\in\overline{\mathbb{Q}}$, then $\pi^{k}E_{k}(z_{0})$
is a period [2].
\end{enumerate}
\end{example}

Though there are numerous non-period transcendental numbers, we have
not a simple criterion for testing them. So the first essential
problem is to find one concrete transcendental number which is not a
period.

It seems that (conjecturally in [2]) the Euler constant
$$\gamma=\sum_{n\rightarrow\infty}(1+\frac{1}{2}+...+\frac{1}{n}-\log
n)=0.5772156...$$ and basis of natural logarithms
$$e=\sum_{n\rightarrow\infty}(1+\frac{1}{n})^{n}=2.7182818...$$ are
not periods.

\section{degree of a period}

Since so many transcendental numbers are periods. How to
differentiate them? To deal with this problem, we introduce the
following concept.
\begin{definition}
If $p$ is a real period, we define the degree of $p$ as the minimal
dimension of the domain $\Sigma$ such that
$$p=\int_{\Sigma}1$$ where $\Sigma$ is a domain in
Euclid space given by polynomial inequalities with algebraic
coefficients.

For any complex period $p=p_{1}+ip_{2}$, we define
$\deg(p)=max(\deg(p_{1}),\deg(p_{2}))$.

\end{definition}

If $p$ is not a period, we may define the $\deg(p)=\infty$. Thus we
can extend the degree to whole complex number field $\mathbb{C}$.

By the definition, $\deg(0)=0$ and $\deg(p)=1$ if and only if $p$ is
an non-zero algebraic number. It is obviously that
$\deg(\pi)=\deg(\log(n))=2$, $n\in\mathbb{Z},n>1$.

Let $\mathbb{P}$ denotes the set of all periods. Let
$P_{k}=\{p\in\mathbb{P}| \deg(p)= k\}$, then
$\mathbb{P}=\bigcup_{k=0}^{\infty}P_{k}$. Thus we give a
classification for all periods.

The following two propositions are the basic properties of degrees.

\begin{proposition}
Let $p_{1},p_{2}$ be two periods, then $\deg(p_{1}p_{2})\leq
\deg(p_{1})+\deg(p_{2})$ and $\deg(p_{1}+p_{2})\leq
max(\deg(p_{1}),\deg(p_{2}))$.
\end{proposition}

\begin{proof}
First we consider the real case.  Assume that $\deg(p_{1})=k,
\deg(p_{2})=l$, then there exists two domains
$\Sigma_{1}\subseteq\mathbb{R}^{k},
\Sigma_{2}\subseteq\mathbb{R}^{l}$ both bounded by polynomial
inequalities with algebraic coefficients such that
$$p_{1}=\int_{\Sigma_{1}}dx_{1}...dx_{k},p_{2}=\int_{\Sigma_{2}}dy_{1}...dy_{l}.$$
One has
$$p_{1}p_{2}=\int_{\Sigma_{1}\times\Sigma_{2}}dx_{1}...dx_{k}dy_{1}...dy_{l},$$
where
$\Sigma_{1}\times\Sigma_{2}\subset\mathbb{R}^{k}\times\mathbb{R}^{l}=\mathbb{R}^{k+l}$
also bounded by polynomial inequalities with algebraic coefficients.
So $\deg(p_{1}p_{2})\leq \deg(p_{1})+\deg(p_{2})$.

Suppose that $k\leq l$, then
$$\deg(p_{1})+\deg(p_{2})=\int_{\Sigma_{k}\times\Delta}dx_{1}...dx_{l}+\int_{\Sigma_{1}}dx_{1}...dx_{l},$$
where $\Delta$ is the $l-k$-times product of $[0,1]$. Hence
$\deg(p_{1}+p_{2})\leq max(\deg(p_{1}),\deg(p_{2}))$.

For the complex case, let $p_{1}=a_{1}+ib_{1}, p_{2}=a_{2}+ib_{2}$,
where $a_{1},a_{2},b_{1},b_{2}$ are real periods. One gets
\begin{eqnarray*}
\deg(p_{1}p_{2})& = & \deg(a_{1}a_{2}-b_{1}b_{2}+i(a_{1}b_{2}+a_{2}b_{1}))\\
               & = & max(\deg(a_{1}a_{2}-b_{1}b_{2}),\deg(a_{1}b_{2}+a_{2}b_{1}))\\
               & \leq & max(max(\deg(a_{1}a_{2}),\deg(b_{1}b_{2})),max(\deg(a_{1}b_{2}),\deg(a_{2}b_{1})))\\
               & = & max(\deg(a_{1}a_{2}),\deg(b_{1}b_{2}),\deg(a_{1}b_{2}),\deg(a_{2}b_{1}))\\
               & \leq & max(\deg(a_{1})+\deg(a_{2}),\deg(b_{1})+\deg(b_{2}),\deg(a_{1})+\deg(b_{2}),\\
               &&\deg(a_{2})+\deg(b_{1}))\\
               & = & \deg(p_{1})+\deg(p_{2}))\\
 \end{eqnarray*}
and
\begin{eqnarray*}
\deg(p_{1}+p_{2})& = & max(\deg(a_{1}+a_{2}),\deg(b_{1}+b_{2}))\\
                & \leq & max(max(\deg(a_{1}),\deg(a_{2})),max(\deg(b_{1}),\deg(b_{2})))\\
                & = & max(\deg(a_{1}),\deg(a_{2}),\deg(b_{1}),\deg(b_{2}))\\
                & = & max(max(\deg(a_{1}),\deg(b_{1})),max(\deg(a_{2}),\deg(b_{2})))\\
                & = & max(\deg(p_{1}),\deg(p_{2})).\\
\end{eqnarray*}
\end{proof}

Generally, we can not get $\deg(p_{1}+p_{2})=
max(\deg(p_{1}),\deg(p_{2}))$. The simplest example is $p_{1}=\pi,
p_{2}=1-\pi$. The following examples also show that the equality
$\deg(p_{1}p_{2})= \deg(p_{1})+\deg(p_{2})$ is not true generally.

\begin{example}
 1):
Consider
$$\xi=\iiint_{x^{2}+y^{2}\leq1,0\leq
z(x^{2}+y^{2}+1)\leq1}dxdydz=\iint_{x^{2}+y^{2}\leq1}\frac{dxdy}{x^{2}+y^{2}+1}
$$ $$=\int_{0}^{2\pi}\int_{0}^{1}\frac{rdrd\theta}{r^{2}+1}=\pi\log2.$$
 $\deg(\xi)\leq3$. But $\deg(\pi)+\deg(\log2)=4$.

2): Consider $$\eta=\iiint_{x^{2}+y^{2}\leq1,0\leq
z((x^{2}+y^{2})^{2}+1)\leq4}dxdydz=\iint_{x^{2}+y^{2}\leq1}\frac{4dxdy}{(x^{2}+y^{2})^{2}+1}
$$ $$=\int_{0}^{2\pi}\int_{0}^{1}\frac{4rdrd\theta}{r^{4}+1}=\pi^{2}.$$
$\deg(\eta)\leq3$. But $2\deg(\pi)=4$.

\end{example}

\begin{proposition}
If $p$ is a nonzero algebraic number and $p_{1}$ is any non-zero
period, then $\deg(p+p_{1})=\deg(p_{1})=\deg(pp_{1})$.
\end{proposition}

\begin{proof}
The first equality follows from $\deg(p_{1})=\deg(-p+p+p_{1})\leq
\deg(p+p_{1})\leq \deg(p_{1})$. For the real case,
$\deg(pp_{1})=\deg(p_{1})$ is obviously from the definition. In
complex case, let $p=a+ib,p_{1}=a_{1}+ib_{1}$,
$a,b\in\overline{\mathbb{Q}}$, $a_{1}$ and $b_{1}$ are any real
periods. We have
\begin{eqnarray*}
\deg(pp_{1})& = & \deg(aa_{1}-bb_{1}+i(ba_{1}+ab_{1}))\\
           & = & max(\deg(aa_{1}-bb_{1})\deg(ba_{1}+ab_{1}))\\
           & \leq & max(max(\deg(aa_{1}),\deg(bb_{1})),max(\deg(ba_{1}),\deg(ab_{1})))\\
           & = & max(\deg(aa_{1}),\deg(bb_{1}),\deg(ba_{1}),\deg(ab_{1})).\\
\end{eqnarray*}
Since $p\neq0$, the last equation equals $\deg(p_{1})$. So
$\deg(pp_{1})\leq \deg(p_{1})$. But $p$ is any nonzero algebraic
number, so one has $\deg(p_{1})=\deg(\frac{1}{p}pp_{1})\leq
\deg(pp_{1})$. Hence $\deg(p_{1})=\deg(pp_{1})$.

\end{proof}

Denote $\mathbb{P}_{k}=\{p\in\mathbb{P}| \deg(p)\leq k\}$,
$\mathbb{P}_{k}+\mathbb{P}_{l}=\{p_{k}+p_{l}|
p_{k}\in\mathbb{P}_{k}, p_{l}\in\mathbb{P}_{l}\}$,
$\mathbb{P}_{k}\mathbb{P}_{l}=\{p_{k}p_{l}| p_{k}\in\mathbb{P}_{k},
p_{l}\in\mathbb{P}_{l}\}$. Then
$\mathbb{P}_{k}+\mathbb{P}_{l}\subseteq\mathbb{P}_{max(k,l)}$ and
$\mathbb{P}_{k}\mathbb{P}_{l}\subseteq\mathbb{P}_{k+l}$.
$\mathbb{P}_{k}$ has a good graded characteristic. It is a additive
group but in general (except $k=1$) not a ring. Proposition 3.4
tells us that $\mathbb{P}_{k}$ is a $\mathbb{P}_{1}$-module, i.e.
$\overline{\mathbb{Q}}$-module.

If we consider the map $d: \mathbb{P}\times
\mathbb{P}\rightarrow\mathbb{Z}$ by $$(p_{1},p_{2})\mapsto
\deg(p_{1}-p_{2}).$$

It obviously satisfies
\begin{itemize}
\item $d(p_{1},p_{2})=0$ if and only if $p_{1}=p_{2}$.
\item $d(p_{1},p_{2})=d(p_{2},p_{1})$.
\item $d(p_{1},p_{2})\leq d(p_{1},p_{3})+d(p_{3},p_{2})$.
\end{itemize}
So $d$ defines a metric on $\mathbb{P}$. The $\mathbb{P}_{k}$ is a
ball of radius $k$ and center at 0.

\section{Some results of periods with low degrees and related problems}

Using the decomposition properties of rational functions with one
variable, we can get the precise forms of some periods with degrees
$\leq2$.

\begin{theorem}
Let $p$ be a real period with $\deg(p)\leq2$. If it can be written
as $p=\int R(x)dx$ for some rational function $R(x)$. Then it has
the form $a\arctan\xi+b\log\eta+c$, where
$a,b,c,\xi,\eta\in\overline{\mathbb{Q}}$.
\end{theorem}

\begin{proof}
Because any rational function can decompose as following four types
$$\frac{A}{x-a},\frac{A}{(x-a)^{n}},\frac{Bx+C}{x^{2}+bx+c},\frac{Bx+C}{(x^{2}+bx+c)^{n}}$$
where $A,B,C,a,b,c\in\overline{\mathbb{Q}}$ and $n\geq2$. By
elementary integral theory, in every type the integral value has the
form $a\arctan\xi+b\log\eta+c$.
\end{proof}

It seems very difficult to determine the degree of a given period.
We present following three problems.

\textbf{Problem 1:} Give a concrete period such that the degree
$\geq3$.

\textbf{Problem 2:} Let $p_{1},p_{2}$ be two non-algebraic periods.
Does $\deg(p_{1}p_{2})\geq2$?

\textbf{Problem 3:} Determine the precise forms of all periods with
degrees $=2$.

\section{Apply the degree to transcendence}

In general, determining the transcendence of the sum of two
transcendental numbers is a very difficult problem. For example, the
transcendence of $e+\pi$ is a longstanding problem in number theory.
But if the transcendental numbers are periods. We have some
theoretic solutions.

\begin{theorem}
Let $p_{1},p_{2}$ be two transcendental periods. If $\deg(p_{1})\neq
\deg(p_{2})$, then both $p_{1}/p_{2}$ and $p_{1}+p_{2}$ are
transcendental numbers.
\end{theorem}

\begin{proof}
If $p_{1}/p_{2}=a$ is algebraic,  by Proposition 3.4 one have
$\deg(p_{1})=\deg(p_{2}a)=\deg(p_{2})$. Which is a contradiction.

Since $\deg(p_{1})\neq \deg(p_{2})$, we may assume that
$\deg(p_{1})<\deg(p_{2})$. By Proposition 3.2 one has
$\deg(p_{2})=\deg(- p_{1}+p_{1}+p_{2})\leq
max(\deg(p_{1}),\deg(p_{1}+p_{2}))=\deg(p_{1}+p_{2})\leq
max(\deg(p_{1}),\deg(p_{2}))=\deg(p_{2})$. So we have
$\deg(p_{1}+p_{2}))=\deg(p_{2})$. Hence $p_{1}+p_{2}$ are
transcendental.
\end{proof}

More generally, we have following result about linearly independence
\begin{theorem}
Let $p_{1},p_{2}$ be any two complex numbers. If $\deg(p_{1})\neq
\deg(p_{2})$, then $p_{1}$ and $p_{2}$ are linearly independent over
$\overline{\mathbb{Q}}$.
\end{theorem}

\begin{proof}
If one is not a period, the theorem is obviously true. We may assume
that both are periods. If $p_{1}$ and $p_{2}$ are linearly
dependent, let $ap_{1}+bp_{2}=c,
a,b,\in\overline{\mathbb{Q}}\setminus0, c\in\overline{\mathbb{Q}}$.
Then $\deg(p_{1})= \deg(\frac{c}{a}-\frac{c}{b}p_{2})=\deg(p_{2})$.
Which is a contradiction.
\end{proof}

It is obviously that above results can extend to arbitrary periods.
That is, if $1<\deg(p_{1})<\deg(p_{2})<...<\deg(p_{k})$, then
$p_{1}+p_{2}+...+p_{k}$ is transcendental. If $1\leq
\deg(p_{1})<\deg(p_{2})<...<\deg(p_{k})\leq\infty$, then
$p_{1},p_{2},...,p_{k}$ are linearly independent over
$\overline{\mathbb{Q}}$.

It was conjectured in [2] that the basis of the natural logarithms
$e$ is not a period. i.e. $\deg(e)=\infty$. This implies that
$e+\pi$ is a transcendental number. Using Theorem 5.2 we can improve
this as

\begin{corollary}
To prove that $e+\pi$ is a transcendental number, one only needs to
prove that $\deg(e)\geq3$.
\end{corollary}

\section{zeta functions of periods}

Let $p$ be a period. We consider the following zeta function for $p$
$$\zeta_{p}(t)=\exp(\sum_{m=1}^{\infty}\frac{t^{m}\deg(p^{m})}{m}), 0\leq t<1.$$
It is the analogue of Weil's zeta function for algebraic variety
over finite fields. We find that $\zeta_{p}(t)$ has some interesting
properties.

\begin{theorem}
\begin{enumerate}
\item $\zeta_{0}(t)=1$. If $p$ is a non-zero algebraic number, then $\zeta_{p}(t)=\frac{1}{1-t}.$

\item $\zeta_{p_{1}p_{2}}(t)\leq\zeta_{p_{1}}(t)\zeta_{p_{2}}(t).$

\item $\zeta_{p}(t)\leq\exp(\frac{t\deg(p)}{1-t}).$

\item If $\deg(p_{1})\leq\deg(p_{2})$, then $\zeta_{p_{1}+p_{2}}(t)\leq\exp(\frac{t\deg(p_{2})}{1-t}).$
\end{enumerate}
\end{theorem}

\begin{proof}
Since $\deg(0)=0$ and $\deg(p)=1$ for non-zero algebraic number $p$.
(1) is directly. From Proposition 3.2, we have
\begin{eqnarray*}
\zeta_{p_{1}p_{2}}(t)& = & \exp(\sum_{m=1}^{\infty}\frac{t^{m}\deg(p_{1}^{m}p_{2}^{m})}{m})\\
                     & \leq & \exp(\sum_{m=1}^{\infty}\frac{t^{m}(\deg(p_{1}^{m})+\deg(p_{2}^{m}))}{m})\\
                     & = & \zeta_{p_{1}}(t)\zeta_{p_{2}}(t),\\
\end{eqnarray*}
and$$\zeta_{p}(t)\leq\exp(\sum_{m=1}^{\infty}t^{m}\deg(p))=\exp(\frac{t\deg(p)}{1-t}).$$

If $\deg(p_{1})\leq\deg(p_{2})$,
\begin{eqnarray*}
\zeta_{p_{1}+p_{2}}(t) & = & \exp(\sum_{m=1}^{\infty}\frac{t^{m}\deg((p_{1}+p_{2})^{m})}{m})\\
                       & \leq & \exp(\sum_{m=1}^{\infty}\frac{t^{m}\deg(p_{1}^{k}p_{2}^{m-k})}{m})\\
                       & \leq & \exp(\sum_{m=1}^{\infty}\frac{t^{m}(k\deg(p_{1})+(m-k)\deg(p_{2}))}{m})\\
                       & \leq & \exp(\sum_{m=1}^{\infty}\frac{t^{m}m\deg(p_{2})}{m})\\
                       & = & \exp(\frac{t\deg(p_{2})}{1-t}).\\
\end{eqnarray*}

In the second step we choose $k$ such that
$\deg(p_{1}^{k}p_{2}^{m-k})=\max\{\deg(p_{1}^{i}p_{2}^{m-i}), 0\leq
i\leq m\}$.

\end{proof}

\textbf{ Problem}: Let $p$ be a non-algebraic period. Is
$\zeta_{p}(t)$ a transcendental function?

\bibliographystyle{amsplain}

\end{document}